\documentclass[12pt]{article}
\usepackage{natbib}

\usepackage{hyperref}

\usepackage{a4wide}
\usepackage{amscd}
\usepackage{enumerate}

\usepackage{graphicx}

\usepackage{amsmath,multirow}
\usepackage{array}
\usepackage{mathptmx}
\usepackage[utf8]{inputenc}
\usepackage{mathtools}
\usepackage{amsfonts}
\usepackage{ed}
\usepackage[mathscr]{eucal}

\usepackage{amsthm}

\theoremstyle{plain}
\newtheorem{theorem}{Theorem}[section]
\newtheorem{remark}{Remark}

\numberwithin{equation}{section}

\newcommand{\rom}[1]{\uppercase\expandafter{\romannumeral #1\relax}}
\begin{document}
{
% A Parallelized Partitioned Approach to Model Massive Non-stationary Non-Gaussian Spatial Datasets
  \title{\bf Distribution of runs and patterns in four state trials}
 \author{Jungtaek Oh\thanks{Corresponding Author:
 Department of Biomedical Science, Kyungpook National University, Daegue, 41566, Republic of Korea
 (e-mail: jungtaekoh0191@gmail.com)}\hspace{.2cm}\\
%Department of Statistics, George Mason University \hspace{.2cm}\\
%Department of Statistics and Data Science, Yonsei University \hspace{.2cm}\\
%Department of Applied Statistics, Yonsei University
}
  \maketitle
} %\fi

\begin{abstract}
From a mathematical and statistical perspective, a segment of a DNA strand can be viewed as a sequence of four-state (A, C, G, T) trials. In thispaper, we consider the distributions of runs and patterns related to the run lengths of four states (A, B, C, and D). Let $X_{1}, X_{2}, \ldots$ be a sequence of four state i.i.d.\ trials taking values in the set $\mathscr{S}=\{A,\ B,\ C,\ D\}$ of four symbols with probability $P(A)=P_{a}$, $P(B)=P_{b}$, $P(C)=P_{c},$ and $P(D)=P_{d}.$ We obtain exact formulas for the probability mass functions of the distributions of runs of B's of order $k$, for longest run, shortest run, waiting time, and run lengths.

\end{abstract}

\noindent%

%\title{Type $\rom{1}$, $\rom{2}$, $\rom{3}$ and $\rom{4}$ $q$-negative binomial distribution of order $k$
%%\footnote{Or: Lancret theorem for Sasakian $3$-manifolds}
%}
%\author{Jungtaek Oh}
%\address{
%Department of Mathematics, Chonnam National University, CNU
%The Institute of Basic Science, Kwangju, 500--757, Korea}
%\email{jtcho@chonnam.ac.kr}
%
%\address{Department of Mathematics Education,
%Utsunomiya University, \newline
%Utsunomiya, 321-8505,
%Japan}
%\email{inoguchi@cc.utsunomiya-u.ac.jp}
%\address{
%Department of Mathematics, Graduate School,
%Chonnam National
%University,
% Kwangju, 500--757, Korea}

%\date{April 10, 2006}
%\pagestyle{empty}

%%%%%%    TEXT START    %%%%%%

\section{Introduction}
\label{Introduction}

Runs and runs-related statistics have been intensively studied in the literature because of their wide range of applications in various areas, including statistics (e.g.,\ hypothesis testing), engineering (e.g.\ system reliability, health services monitoring and quality control), molecular biology and bioinformatics (e.g.,\ population genetics and deoxyribonucleic acid (DNA) sequence homology), physics, psychology, radar astronomy, computer science (e.g.,\ encoding/decoding and transmission of digital information), and finance (e.g.,\ financial engineering, risk analysis, and risk prediction). Significant progress on runs and related statistics in the past few decades has been comprehensively surveyed by \citet{balakoutras03} and \citet{fuwendy03}. Moreover, there are various more recent contributions on the topic, such as \citet{arapis18}, \citet{serkan18}, \citet{kong19}, \citet{makri19}, and \citet{aki19}.\\

There are two main types of problems concerning runs and runs-related statistics.
\noindent
\begin{itemize}
  \item[(1)] The number of trials until the first occurrence of a run or pattern (or until the $r$-th occurrence of a run or pattern), which is called a waiting time problem.
  \item[(2)] The number of occurrences of a run or pattern until the $n$-th trial.
\end{itemize}
\noindent
Waiting time distributions have attracted considerable interest in applied probability. Properties of waiting time distributions are extensively analyzed in \citet{ebneshahrashoob90}, \citet{aki92}, \citet{balasubramanian93}, \citet{koutras96}, \citet{antzoulakos97}, \citet{antzoulakos99}, \citet{kim13}, and \citet{inoue14}.\\
\noindent
The geometric distribution of order $k$ is a well-known waiting time distribution. This is defined as the distribution of the number of trials until the first consecutive $k$ successes. This definition is based on \citet{philippou83}. This distribution becomes the classical \textrm{geometric distribution} for the case $k=1$ (the number of Bernoulli trials needed to get one success) with probability mass function $f(x)=q^{x-1}p$ $(x\geq 1).$\\
\noindent
Let $X_{1}, X_{2}, \ldots$ be a sequence of four state independent identically distributed (i.i.d.) trials that take values in the set $\mathscr{S}=\{A,\ B,\ C,\ D\}$ of four symbols. For the number of trials $W_1^k$ until the first consecutive $k$ successes we have the following equivalent definitions of $W_{1}^{k}$:
  \begin{equation*}\label{eq: 1.1}
\begin{split}
 W_{1}^{k} \ =& \ \text{min}\{n : X_{n-k+1} =\cdots = X_{n} = 1\}\\
 =& \ \text{min}\left\{n : \prod_{n-k+1}^{n}X_{j} = 1\right\}\\
  =& \ \text{min}\left\{n : \sum_{n-k+1}^{n}X_{j} = k\right\}.
\end{split}
\end{equation*}
\noindent
Various statistical properties of the geometric distribution of order $k$ have been applied to multiple areas, ranging from quality control to reliability. In particular, the consecutive-$k$-out-of-$n$:$F$ system (see, \citet{chao1995}, \citet{chang2000}, \citet{kuo2003}, \citet{ser10}, and \citet{triantafyllou2015}) is closely related to this distribution. Another closely related random variable is the length $L_{n}$ of the longest success run in $n$ four state trials. Because $W_{1}^{k}\leq n$ if and only if $L_{n} \geq k$, we have the following between the probability distribution functions of $W_{1}^{k}$ and $L_{n}$ : $\mathbf{P(}W_{1}^{k} \leq n) = \mathbf{P}(L_{n} \geq k).$\\
\noindent
Let $X_{1}, X_{2}, \ldots, X_{n}$ be random variables i.i.d. $W_{1}^{k}$. The distribution of the number $W_{r}^{k}$ of trials until the $r$-th appearance of a success run of length $k$ is a negative binomial distribution of order $k.$ This follows from the fact that it is the distribution of the $r$-fold convolution of the geometric distribution of order $k$ :
\begin{equation*}
W_{r}^{k}=\sum_{i=1}^{r}X_{i}.
\end{equation*}
\noindent
Some main aspects of the occurrences of runs and patterns are: where, when, and how many times they occur. To study these, statistics that count runs and patterns according to various enumerating schemes are defined. \citet{feller68} presented a classical counting scheme for enumerating runs of fixed length : once $k$ consecutive successes occur, the number of occurrences of $k$ consecutive successes is incremented and the counting procedure starts anew, referred to as a \textit{nonoverlapping} counting scheme, which follows a binomial distribution of order $k$. In another counting scheme, we count only the number $E_n^k$ of success runs of length exactly $k$ preceded and succeeded either by failures or by nothing (\citet{mood40}).\\
\noindent
In a sequence of $n$ four state trials, we can define other statistics via runs and patterns, such as the longest (maximal) and the shortest (minimal) run lengths (see, \cite{sen02}). The longest and shortest runs are the upper and lower bounds of the number of consecutive successes that appear in a sequence of $n$ four state trials, respectively. These distributions can be applied in DNA-type sequence comparison by observing the frequencies of the longest runs of matches or mismatches.\\
\noindent
Sequences of categorical outcomes frequently arise in biomedical research, representing, for example, DNA strand segments or outcomes of healthcare evaluations. They are typically analyzed by defining problem-specific statistics involving runs and patterns. A DNA molecule is a chain or sequence of pairs of nucleotides with four base structures: adenine, cytosine, guanine, and thymine (or A, C, G, and T, respectively). The occurrence of a specified nucleotides sequence in some portion of the chain is the event that the specified run of A, C, G, and T occurs. Mathematically, a random DNA strand segment can be viewed as a sequence of four-state (A, C, G, T) trials. We consider some distributions of runs and patterns related to run lengths for multistate trials, particularly four-state trials.\\
\noindent
For a sequence $X_{1}, X_{2}, \ldots, X_{n}$ of i.i.d.\ trials taking values in $\mathscr{S}=\{A,\ B,\ C,\ D\}$ with probabilities $P(A)=P_{a}$, $P(B)=P_{b}$, $P(C)=P_{c}$, and $P(D)=P_{d},$ such that $P_{a}+P_{b}+P_{c}+P_{d}=1,$ we consider the following stochastic variables.

\begin{itemize}
  \item Let $N_{n}^{k}$ denote the number of nonoverlapping runs of B's of order $k$ in $n$ independent trials.
  \item Let $E_{n}^{k}$ denote the number of runs of B's of exact length $k$, preceded and succeeded by A, C, D, or nothing.
  \item Let $W_{1}^{k}$ denote the waiting time for the first occurrence of a run of B's of length $k$.
  \item Let $W_{r}^{k}$ denote the waiting time for the $r$-th occurrence of the run of B's of length $k$.
  \item Let $L_{n}$ denote the maximum length of a run of B's (longest run statistics).
  \item Let $M_{n}$ denote the minimum length of a run of B's (shortest run statistics).
  \item Let $NL_{n}$ denote the number of times a nonoverlapping run of length $L_{n}$ appears.
  \item Let $NM_{n}$ denote the number of times a nonoverlapping run of length $M_{n}$ appears.
\end{itemize}
\noindent
To illustrate the a aforementioned variables, we consider the runs of B's in the following example for $n=40$: $$DAAABBBBBCAABBBCCADBBBBCCDABBBBBBCACCDBB.$$
 For which one can check $N_{40}^{2}=9$, $N_{40}^{3}=5$, $L_{40}=6$, $M_{40}=2$, $E_{40}^{2}=1$, $E_{40}^{3}=1$, $E_{40}^{4}=1$, $E_{40}^{5}=1$, $E_{40}^{6}=1$,
 $NL_{40}=1$ and $NM_{40}=1$.\\
\noindent
In this paper, we obtain the probability distribution functions of runs and patterns in four state i.i.d. trials in terms of multinomial coefficients. The exact PMFs were are derived via combinatorial analysis. In Section \ref{section2}, we obtain the PMFs of a binomial distribution of order $k$ and distribution of runs of length exactly $k$. In Section \ref{section3}, we obtain PMF's of a geometric distribution of order $k$ and a negative binomial distribution of order $k$. In Section \ref{section4} we obtain the PMF's for the distribution of the longest and shortest runs.

\section{Discrete distribution of order $k$}\label{section 2-1}
\label{section2}

We consider a sequence $X_{1}, X_{2}, \ldots, X_{n}$ of multistate trials defined on the state space $\mathscr{S}=\{A,\ B,\ C,\ D\}$ with probabilities $P(A)=P_{a}$, $P(B)=P_{b}$, $P(C)=P_{c}$ and $P(D)=P_{d},$ such that $P_{a}+P_{b}+P_{c}+P_{d}=1,$.
In this section, we obtain the PMFs of a binomial distribution of order $k,$ and the distribution of runs of length exactly $k$ using combinatorial analysis.
\subsection{Binomial distribution of order $k$}
First, we consider the number $N_{n}^{k}$ of runs of B's of length $k$ in $n$ i.i.d.\ four state trials. We establish the PMF of the random variable $N_{n}^{k}$ using combinatorial analysis.

\begin{theorem}
\label{thm 1}
The {\rm PMF} of $N_{n}^{k}$, for $ 0\leq x\leq \left[\frac{n}{k}\right]$, is given by
\begin{equation*}\label{eq: 1.1}
\begin{split}
P(N_{n}^{k}=x)=P_{b} ^{n}\sum_{i=0}^{k-1} \sum_{\star i}{\sum_{t=1}^{k} (x_{t}+ y_{t}+z_{t})+x   \choose   x_{1},\ldots,x_{k},y_{1},\ldots,y_{k},z_{1},\ldots,z_{k},x}\left(\frac{P_{a}}{P_{b}}\right)^{\sum_{t=1}^{k} x_{t}} \left(\frac{P_{c}}{P_{b}}\right)^{\sum_{t=1}^{k} y_{t}} \left(\frac{P_{d}}{P_{b}}\right)^{\sum_{t=1}^{k} z_{t}},
\end{split}
\end{equation*}
where the inner summation $\sum_{\star}$ is over all nonnegative integers $x_{1},\ldots,x_{k}, y_{1},\ldots,y_{k}, z_{1},\ldots,z_{k}$ for which
$\sum_{t=1}^{k} t(x_{t}+y_{t}+z_{t})+kx+i=n,$ for $ i=0, 1,\ldots , k-1.$
\end{theorem}
\begin{proof}
Let \ ${\underbrace{B\cdots B}_{t-1}}A=O_{t}$,\quad ${\underbrace{B\cdots B}_{t-1}}C=J_{t}$,\quad ${\underbrace{B\cdots B}_{t-1}}D=T_{t}$,\ where \ $1\leq t\leq k$.
A typical element of the event $\left\{N_{n}^{k}=x\right\}$ is a sequence
\begin{equation*}\label{eq: 1.1}
\begin{split}
\cdots {\boxed{ R_{t}}} \cdots \overset{1}{\boxed{{\underbrace{B\dots B}_{k}}}} \cdots {\boxed{ R_{t}}} \cdots \overset{2}{\boxed{{\underbrace{B\ldots B}_{k}}}} \cdots {\boxed{ R_{t}}} \cdots \overset{x}{\boxed{{\underbrace{B\ldots B}_{k}}}} \cdots {\boxed{ R_{t}}} \cdots\underbrace{B\ldots B}_{i}
\end{split},
\end{equation*}
\noindent
where $i\in \{0,\ldots,k-1\}$, and ${\boxed{ R_{t}}}$ represents any combination of the strings $O_{t},J_{t}$, and $T_{t}$ appearing altogether $x_{t}$, $y_{t}$, and $z_{t}$ times, respectively, in the sequence, satisfying
\begin{equation}\label{eq:bn}
\begin{split}
\sum_{t=1}^{k} t(x_{t}+y_{t}+z_{t})+kx+i=n.
\end{split}
\end{equation}
\noindent
The number of different ways of arranging this sequence equals
\noindent
\begin{equation}\label{eq:bn}
\begin{split}
{\sum_{t=1}^{k} x_{t}+\sum_{t=1}^{k} y_{t}+\sum_{t=1}^{k} z_{t}+x   \choose   x_{1},\ldots,x_{k},y_{1},\ldots,y_{k},z_{1},\ldots,z_{k}, x}.
\end{split}
\end{equation}
\noindent
Because of the independence of the trials, the probability of the above sequence is
\noindent
\begin{equation}\label{eq:bn}
\begin{split}
P_{a}^{\sum_{t=1}^{k} x_{t}} P_{b}^{\sum_{t=1}^{k} (t-1)(x_{t}+y_{t}+z_{t})+kx} P_{c}^{\sum_{t=1}^{k} y_{t}} P_{d}^{\sum_{t=1}^{k} z_{t}}.
\end{split}
\end{equation}
\noindent
The probability of the run of B's of length $i$ at the end of each of these sequences is $P_{b}^{i}$ $(0\leq i< k)$, which leads to the overall probability
\noindent
\begin{equation*}\label{eq: 1.1}
\begin{split}
P_{a}^{\sum_{t=1}^{k} x_{t}} P_{b}^{\sum_{t=1}^{k} (t-1)(x_{t}+y_{t}+z_{t})+kx+i} P_{c}^{\sum_{t=1}^{k} y_{t}} P_{d}^{\sum_{t=1}^{k} z_{t}}.
\end{split}
\end{equation*}
\noindent
Summing over $i=0,1,\ldots,k-1$ the result follows.
\end{proof}

\begin{remark}

For $P_{a}=q$, $P_{b}=p$ such that $p+q=1$, $P_{c}=P_{d}=0$ and $y_{1},\ldots,y_{k},z_{1},\ldots,z_{k}=0$, Theorem~\ref{thm 1} reduces to Theorem~2.1 of \citet{philippou86}.
\end{remark}

\subsection{Distributions of runs of length exactly $k$}
In this subsection we consider the number $E_{n}^{k}$ of runs of B's of length $k$ in $n$ i.i.d.\ four state trials. We establish the PMF of the random variable $E_{n}^{k}$ using combinatorial analysis.
\noindent
\begin{theorem}
\label{thm 2}
The PMF of $E_n^k$, for $0\leq x\leq \left[\frac{n+1}{k+1}\right]$, is given by
\begin{equation*}\label{eq: 1.1}
\begin{split}
P(E_{n}^{k}=x)=&P_{b}^{n}\sum_{i=0}^{n-x(k+1)}\sum_{\star (k,i)}\bigg[\text{\small$
 {x_{1}+\cdots +x_{n-x(k+1)}+y_{1}+\cdots+y_{n-x(k+1)}+z_{1}+\cdots+z_{n-x(k+1)}   \choose   x_{1},\ldots,x_{n-x(k+1)},y_{1},\ldots,y_{n-x(k+1)},z_{1},\ldots,z_{n-x(k+1)}}$}\\
&\text{\small$\times
\left(\frac{P_{a}}{P_{b}}\right)^{x_{1}+\cdots +x_{k+1}+\cdots +x_{n-x(k+1)}} \left(\frac{P_{c}}{P_{b}}\right)^{y_{1}+\cdots+y_{k+1}+\cdots+y_{n-x(k+1)}} \left(\frac{P_{b}}{P_{d}}\right)^{z_{1}+\cdots+z_{k+1}+\cdots+z_{n-x(k+1)}}$}\bigg],
\end{split}
\end{equation*}
\noindent
where the inner summation $\sum_{\star (k,i)}$ is over all nonnegative integers $ x_{1}$, $\ldots$, $x_{n-x(k+1)}$, $y_{1}$, $\ldots$, $y_{n-x(k+1)}$, $z_{1}$, $\ldots$, $z_{n-x(k+1)}$ for which $\sum_{t=1}^{n-x(k+1)}t(x_{t}+y_{t}+z_{t})=n-i$ and
\begin{equation*}
x_{k+1}+y_{k+1}+z_{k+1}=
\left\{
  \begin{array}{ll}
    x & \text{if $i\neq k,$} \\
    x-1 & \text{if $i=k$.}\\
  \end{array}
\right.
\end{equation*}
\end{theorem}

\begin{proof}
Let  ${\underbrace{B\cdots B}_{t-1}}A=O_{t}$,\quad ${\underbrace{B\cdots B}_{t-1}}C=J_{t}$,\quad ${\underbrace{B\cdots B}_{t-1}}D=T_{t}$,\ for $t=1, 2, \ldots, n-x(k+1)$. A typical element of the event $\left\{E_n^k=x\right\}$ is a sequence
\begin{equation*}\label{eq: 1.1}
\begin{split}
\cdots {\boxed{ R_{t}}} \cdots \overset{1}{\boxed{O_{k+1}\ or\  J_{k+1}\ or\ T_{k+1}}} \cdots  {\boxed{ R_{t}}} \cdots\overset{x}{\boxed{O_{k+1}\ or\ J_{k+1}\ or\ T_{k+1}}} \cdots {\boxed{ R_{t}}} \cdots\underbrace{B\ldots B}_{i}
\end{split},
\end{equation*}
\noindent
 where $0\leq i\leq n-x(k+1)$, and ${\boxed{ R_{t}}}$ represents any string $O_{t}$, $J_{t}$, and $T_{t}$ appearing $x_{t}$, $y_{t}$, and $z_{t}$ times, respectively, in the sequence, satisfying
 \noindent
 \begin{equation}\label{eq: 1.1}
\sum_{t=1}^{n-x(k+1)}t(x_{t}+y_{t}+z_{t})=n-i,
\end{equation}
subject to the condition $x_{k+1}+y_{k+1}+z_{k+1}=
\left\{
  \begin{array}{ll}
    x & \text{if $i\neq k,$} \\
    x-1 & \text{if $i=k$.}\\
  \end{array}
\right.$
\noindent
The number of different ways of arranging the sequence equals
\begin{equation*}\label{eq:bn}
\begin{split}
 {x_{1}+\cdots +x_{n-x(k+1)}+y_{1}+\cdots+y_{n-x(k+1)}+z_{1}+\cdots+z_{n-x(k+1)}   \choose   x_{1},\ldots,x_{n-x(k+1)},y_{1},\ldots,y_{n-x(k+1)},z_{1},\ldots,z_{n-x(k+1)}}.
\end{split}
\end{equation*}
\noindent
Because of the independence of the trials, the sequence has probability

\begin{equation*}\label{eq:bn}
\begin{split}
&P_{a}^{x_{1}+\cdots +x_{n-x(k+1)}} P_{c}^{y_{1}+\cdots +y_{n-x(k+1)}}P_{d}^{z_{1}+\cdots +z_{n-x(k+1)}}\\
&\times P_{b}^{(x_{2}+y_{2}+z_{2})+2(x_{3}+y_{3}+z_{3})+\cdots +\{n-x(k+1)-1\}(x_{n-x(k+1)}+y_{n-x(k+1)}+z_{n-x(k+1)})}.
\end{split}
\end{equation*}
\noindent
The probability of the run of B's of length $i\in \{0,1,\ldots,n-x(k+1)\}$ at the end of each of these sequences is $P_{b}^{i}$, which leads to the overall probability
\noindent
\begin{equation*}\label{eq: 1.1}
\begin{split}
&P_{a}^{x_{1}+\cdots +x_{n-x(k+1)}} P_{c}^{y_{1}+\cdots +y_{n-x(k+1)}}P_{d}^{z_{1}+\cdots +z_{n-x(k+1)}}\\
&\times P_{b}^{(x_{2}+y_{2}+z_{2})+2(x_{3}+y_{3}+z_{3})+\cdots +\{n-x(k+1)-1\}(x_{n-x(k+1)}+y_{n-x(k+1)}+z_{n-x(k+1)})+i}.
\end{split}
\end{equation*}
\noindent
Summing over $i=0,1,\ldots,n-x(k+1)$, the result follows.
\end{proof}

\section{Waiting time distributions}
\label{section3}
 We consider a sequence $X_{1}, X_{2}, \ldots$ of multistate trials defined on the state space $\mathscr{S}=\{A,\ B,\ C,\ D\}$ with probabilities $P(A)=P_{a}$, $P(B)=P_{b}$, $P(C)=P_{c}$, and $P(D)=P_{d},$ such that $P_{a}+P_{b}+P_{c}+P_{d}=1$.
In this section, we obtain PMFs for the geometric distribution of order $k$ and the negative binomial distribution of order $k$, by employing combinatorial analysis.

\subsection{Geometric distribution of order $k$}
 First, we consider the waiting time $W_{1}^{k}$ for the first occurrence of a run of B's of length $k$. We establish the PMF of the random variable $W_{1}^{k}$ using combinatorial analysis.
\begin{theorem}
\label{thm3}
The PMF of $W_{1}^{k}$ is given by
\begin{equation*}\label{eq: 1.1}
\begin{split}
P(W_{1}^{k}=n)=P_{b} ^{n}\sum_{\star}{\sum_{t=1}^{k}(x_{t}+y_{t}+z_{t})   \choose   x_{1},\ldots,x_{k},y_{1},\ldots,y_{k},z_{1},\ldots,z_{k}}\left(\frac{P_{a}}{P_{b}}\right)^{\sum_{t=1}^{k} x_{t}} \left(\frac{P_{c}}{P_{b}}\right)^{\sum_{t=1}^{k} y_{t}} \left(\frac{P_{d}}{P_{b}}\right)^{\sum_{t=1}^{k} z_{t}},
\end{split}
\end{equation*}
where the inner summation $\sum_{\star}$ is over all nonnegative integers $x_{1},\ldots,x_{k}, y_{1},\ldots,y_{k}, z_{1},\ldots,z_{k}$ for which
$\sum_{t=1}^{k} t(x_{t}+y_{t}+z_{t})=n-k$.
\end{theorem}

\begin{proof}
Let  ${\underbrace{B\cdots B}_{t-1}}A=O_{t}$, ${\underbrace{B\cdots B}_{t-1}}C=J_{t}$, ${\underbrace{B\cdots B}_{t-1}}D=T_{t}$, $t=1,\ldots,k$. A typical element of the event $\left\{W_{1}^{k}=x\right\}$ is a sequence
\noindent
\begin{equation*}
\cdots \cdots\cdots\cdots {\boxed{ R_{t}}} \cdots \cdots\cdots\cdots  {\boxed{{\underbrace{B\ldots B}_{k}}}},
\end{equation*}
\noindent
where ${\boxed{ R_{t}}}$ represents any string $O_{t}$, $J_{t}$, and $T_{t}$ appearing $x_{t}$, $y_{t}$, and $z_{t}$ times, respectively, in the sequence, satisfying $\sum_{t=1}^{k} t(x_{t}+y_{t}+z_{t})=n-k$.
The number of different ways of arranging the sequence equals
\begin{equation}\label{eq:bn}
\begin{split}
{\sum_{t=1}^{k} (x_{t}+y_{t}+z_{t})   \choose   x_{1},\ldots,x_{k},y_{1},\ldots,y_{k},z_{1},\ldots,z_{k}}.\end{split}
\end{equation}
\noindent
Because of the independence of the trials, the probability of the sequence is
\noindent
\begin{equation}\label{eq:bn}
\begin{split}
P_{a}^{\sum_{t=1}^{k} x_{t}} P_{b}^{\sum_{t=1}^{k} (t-1)(x_{t}+y_{t}+z_{t})} P_{c}^{\sum_{t=1}^{k} y_{t}} P_{d}^{\sum_{t=1}^{k} z_{t}}.
\end{split}
\end{equation}
\noindent
The probability of the run of B's of length $k$ at the end of each these sequences is $P_{b}^{k}$, which leads to the overall probability
\begin{equation*}\label{eq: 1.1}
\begin{split}
P_{a}^{\sum_{t=1}^{k} x_{t}} P_{b}^{\sum_{t=1}^{k} (t-1)(x_{t}+y_{t}+z_{t})+k} P_{c}^{\sum_{t=1}^{k} y_{t}} P_{d}^{\sum_{t=1}^{k} z_{t}}.
\end{split}
\end{equation*}
\end{proof}

\begin{remark}

For $P_{a}=q$, $P_{b}=p$ such that $p+q=1$, $P_{c}=P_{d}=0$ and $y_{1},\ldots,y_{k},z_{1},\ldots,z_{k}=0$, Theorem \ref{thm3} reduces to Theorem 3.1 of \citet{philippou82}.
\end{remark}

%\begin{equation*}\label{eq: 1.1}
%\begin{split}
%P&(W_{1}^{k}=n)\\
%=&\sum_{x_{1},\cdots,x_{k},y_{1},\cdots,y_{k},z_{1},\cdots,z_{k}}\Big[{\sum_{t=1}^{k} (x_{t}+y_{t}+z_{t})   \choose   x_{1},\cdots,x_{k},y_{1},\cdots,y_{k},z_{1},\cdots,z_{k}}\\
%&\times P_{a}^{\sum_{t=1}^{k} x_{t}} P_{b}^{\sum_{t=1}^{k} (t-1)(x_{t}+y_{t}+z_{t})} P_{c}^{\sum_{t=1}^{k} y_{t}} P_{d}^{\sum_{t=1}^{k} z_{t}} P_{b}^{k}\Big]\\
%=&P_{b} ^{n}\sum_{x_{1},\cdots,x_{k},y_{1},\cdots,y_{k},z_{1},\cdots,z_{k}}\Big[{\sum_{t=1}^{k} x_{t}+y_{t}+z_{t}   \choose   x_{1},\cdots,x_{k},y_{1},\cdots,y_{k},z_{1},\cdots,z_{k}}\\
%&\times\Big(\frac{P_{a}}{P_{b}}\Big)^{\sum_{t=1}^{k} x_{t}} \Big(\frac{P_{c}}{P_{b}}\Big)^{\sum_{t=1}^{k} y_{t}} \Big(\frac{P_{d}}{P_{b}}\Big)^{\sum_{t=1}^{k} z_{t}}\Big].\\
%\end{split}
%\end{equation*}

\subsection{Negative bimomial distribution of order $k$}
Let $W_{r}^{k}$ be the random variable denoting the waiting time for the $r$-th occurrence of a run of B's of length $k$. We establish the PMF of random variables $W_{r}^{k}$ using combinatorial analysis.

\begin{theorem}
\label{thm4}
The PMF of $W_{r}^{k}$ in $n$ four state i.i.d. trials is given by
\begin{equation*}\label{eq: 1.1}
\begin{split}
P(W_{r}^{k}=n)=P_{b} ^{n}\sum_{\star}{\sum_{t=1}^{k} (x_{t}+y_{t}+z_{t})+r-1   \choose   x_{1},\ldots,x_{k},y_{1},\ldots,y_{k},z_{1},\ldots,z_{k},r-1}\left(\frac{P_{a}}{P_{b}}\right)^{\sum_{t=1}^{k} x_{t}} \left(\frac{P_{c}}{P_{b}}\right)^{\sum_{t=1}^{k} y_{t}} \left(\frac{P_{d}}{P_{b}}\right)^{\sum_{t=1}^{k} z_{t}},
\end{split}
\end{equation*}
where the inner summation $\sum_{\star}$ is over all nonnegative integers $x_{1},\ldots,x_{k}$, $y_{1},\ldots,y_{k}$, $z_{1},\ldots,z_{k}$ for which
$\sum_{t=1}^{k} t(x_{t}+y_{t}+z_{t})+kr=n$.
\end{theorem}

\begin{proof}
Let  ${\underbrace{B\cdots B}_{t-1}}A=O_{t}$, ${\underbrace{B\cdots B}_{t-1}}C=J_{t}$, ${\underbrace{B\cdots B}_{t-1}}D=T_{t}$, $t=1,\ldots,k$. A typical element of the event $\left\{W_{r}^{k}=x\right\}$ is a sequence
\noindent
\begin{equation*}
\cdots {\boxed{ R_{t}}} \cdots \overset{1}{\boxed{{\underbrace{B\dots B}_{k}}}} \cdots {\boxed{ R_{t}}} \cdots \overset{2}{\boxed{{\underbrace{B\ldots B}_{k}}}} \cdots {\boxed{ R_{t}}} \cdots \overset{r-1}{\boxed{{\underbrace{B\ldots B}_{k}}}} \cdots {\boxed{ R_{t}}} \cdots\overset{r}{\boxed{{\underbrace{B\ldots B}_{k}}}},
\end{equation*}
\noindent
where ${\boxed{ R_{t}}}$ represents any string $O_{t}$, $J_{t}$, and $T_{t}$ appearing $x_{t}$, $y_{t}$, and $z_{t}$ times, respectively, in the sequence, satisfying $\sum_{t=1}^{k} t(x_{t}+y_{t}+z_{t})+kr=n$.
The number of different ways of arranging the sequence equals
\noindent
\begin{equation}\label{eq:bn}
\begin{split}
{\sum_{t=1}^{k} (x_{t}+y_{t}+z_{t})+r-1   \choose   x_{1},\ldots,x_{k},y_{1},\ldots,y_{k},z_{1},\ldots,z_{k},r-1}.
\end{split}
\end{equation}
\noindent
Because of the independence of the trials, the probability of the sequence is
\noindent
\begin{equation}\label{eq:bn}
\begin{split}
P_{a}^{\sum_{t=1}^{k} x_{t}} P_{b}^{\sum_{t=1}^{k} (t-1)(x_{t}+y_{t}+z_{t})+k(r-1)} P_{c}^{\sum_{t=1}^{k} y_{t}} P_{d}^{\sum_{t=1}^{k} z_{t}}.
\end{split}
\end{equation}
\noindent
The probability of the run of B's of length $k$ at the end of each of these sequences is $P_{b}^{k}$, which leads to the overall probability
\noindent
\begin{equation*}\label{eq: 1.1}
\begin{split}
P_{a}^{\sum_{t=1}^{k} x_{t}} P_{b}^{\sum_{t=1}^{k} (t-1)(x_{t}+y_{t}+z_{t})+kr} P_{c}^{\sum_{t=1}^{k} y_{t}} P_{d}^{\sum_{t=1}^{k} z_{t}}.
\end{split}
\end{equation*}
\end{proof}

\begin{remark}
For $P_{a}=q$, $P_{b}=p$ such that $p+q=1$, $P_{c}=P_{d}=0$ and $y_{1},\ldots,y_{k},z_{1},\ldots,z_{k}=0$, Theorem \ref{thm4} reduces to Theorem 3.1 (a) of \citet{philippou84}.
\end{remark}

%\begin{equation*}\label{eq: 1.1}
%\begin{split}
%P&(W_{r}^{k}=n)\\
%=&\sum_{x_{1},\ldots,x_{k},y_{1},\ldots,y_{k},z_{1},\ldots,z_{k}}\Big[{\sum_{t=1}^{k} (x_{t}+y_{t}+z_{t})+r-1   \choose   x_{1},\ldots,x_{k},y_{1},\ldots,y_{k},z_{1},\ldots,z_{k},r-1}\\
%&\times P_{a}^{\sum_{t=1}^{k} x_{t}} P_{b}^{\sum_{t=1}^{k} (t-1)(x_{t}+y_{t}+z_{t})+rk} P_{c}^{\sum_{t=1}^{k} y_{t}} P_{d}^{\sum_{t=1}^{k} z_{t}}\Big] \\
%=&P_{b} ^{n}\sum_{x_{1},\ldots,x_{k},y_{1},\ldots,y_{k},z_{1},\ldots,z_{k}}\Big[{\sum_{t=1}^{k} (x_{t}+y_{t}+z_{t})+r-1   \choose   x_{1},\ldots,x_{k},y_{1},\ldots,y_{k},z_{1},\ldots,z_{k},r-1}\\
%&\times\Big(\frac{P_{a}}{P_{b}}\Big)^{\sum_{t=1}^{k} x_{t}} \Big(\frac{P_{c}}{P_{b}}\Big)^{\sum_{t=1}^{k} y_{t}} \Big(\frac{P_{d}}{P_{b}}\Big)^{\sum_{t=1}^{k} z_{t}}\Big].\\
%\end{split}
%\end{equation*}

\section{Distributions of run lengths}\label{section4}

We consider a sequence $X_{1}, X_{2}, \ldots, X_{n}$ of multistate trials defined on the state space $\mathscr{S}=\{A,\ B,\ C,\ D\}$ with probabilities $P(A)=P_{a}$, $P(B)=P_{b}$, $P(C)=P_{c}$ and $P(D)=P_{d},$ such that $P_{a}+P_{b}+P_{c}+P_{d}=1,$. In this section, we obtain the PMFs of the distribuions of the longest and shortest runs and establish the PMFs and their joint distributions using combinatorial analysis.

\subsection{Distribution of the longest run length}

Let $L_{n}$ be the maximum length of a run of B's in $n$ four state i.i.d.\ trials, which is called longest run statistics. We establish the PMF of the random variable $L_{n}$ using combinatorial analysis.

\begin{theorem}
\label{thm 5}
The PMF of $L_{n}$ is given by
\begin{equation*}\label{eq: 1.1}
\begin{split}
P(L_{n}=\ell)=P_{b} ^{n}\sum_{i=0}^{\ell}\sum_{\star}{\sum_{t=1}^{\ell+1} (x_{t}+y_{t}+z_{t})  \choose   x_{1},\ldots,x_{\ell+1},y_{1},\ldots,y_{\ell+1},z_{1},\ldots,z_{\ell+1}}\left(\frac{P_{a}}{P_{b}}\right)^{\sum_{t=1}^{\ell+1} x_{t}} \left(\frac{P_{c}}{P_{b}}\right)^{\sum_{t=1}^{\ell+1} y_{t}} \left(\frac{P_{d}}{P_{b}}\right)^{\sum_{t=1}^{\ell+1} z_{t}},
\end{split}
\end{equation*}
\noindent
where the inner summation $\sum_{\star}$ is over all nonnegative integers $x_{1},\ldots,x_{\ell+1}$, $y_{1},\ldots,y_{\ell+1}$, $z_{1},\ldots,z_{\ell+1}$ such that $\sum_{t=1}^{\ell+1} t(x_{t}+y_{t}+z_{t}) =n-i$, for $0 \leq i\leq \ell,$ and satisfying at least one of the conditions $x_{\ell+1}\geq 1$, $y_{\ell+1}\geq 1$, $z_{\ell+1}\geq 1$, and $i=\ell.$
\end{theorem}
\begin{proof}
Let  ${\underbrace{B\cdots B}_{t-1}}A=O_{t}$,\quad ${\underbrace{B\cdots B}_{t-1}}C=J_{t}$,\quad ${\underbrace{B\cdots B}_{t-1}}D=T_{t}$,\ where $t=1,\ldots,$ $\min(\ell+1,n)$.
A typical element of the event $\left\{L_{n}=\ell\right\}$ is a sequence
\begin{equation*}\label{eq: 1.1}
\begin{split}
\cdots {\boxed{ R_{t}}} \cdots\underbrace{B\cdots B}_{i},
\end{split}
\end{equation*}
 where ${\boxed{ R_{t}}}$ represents any of the strings $O_{t}$, $J_{t}$, and $T_{t}$ appearing $x_{t}$, $y_{t}$, and $z_{t}$ times, respectively, in the sequence, satisfying $\sum_{t=1}^{\ell+1} t(x_{t}+y_{t}+z_{t}) =n-i$, for $0\leq i\leq \ell,$ and satisfying at least one of the conditions: $x_{\ell+1}\geq 1$, $y_{\ell+1}\geq 1$, $z_{\ell+1}\geq 1$, and $i=\ell.$
\noindent
The number of different ways of arranging the sequence equals
\noindent
\begin{equation}\label{eq:bn}
\begin{split}
{\sum_{t=1}^{\ell+1} (x_{t}+y_{t}+z_{t})   \choose   x_{1},\ldots,x_{\ell+1},y_{1},\ldots,y_{\ell+1},z_{1},\ldots,z_{\ell+1}}.
\end{split}
\end{equation}
\noindent
Because of the independence of the trials, the probability of the above sequence is
\noindent
\begin{equation}\label{eq:bn}
\begin{split}
P_{a}^{\sum_{t=1}^{\ell+1} x_{t}} P_{b}^{\sum_{t=1}^{\ell+1} (t-1)(x_{t}+y_{t}+z_{t})} P_{c}^{\sum_{t=1}^{\ell+1} y_{t}} P_{d}^{\sum_{t=1}^{\ell+1} z_{t}}.
\end{split}
\end{equation}
\noindent
The probability of the run of B's of length $i$ $(0\leq i\leq \ell)$ at the end of each of these sequences is $P_{b}^{i}$, which leads to the overall probability
\begin{equation*}\label{eq: 1.1}
\begin{split}
P_{a}^{\sum_{t=1}^{\ell+1} x_{t}} P_{b}^{\sum_{t=1}^{\ell+1} (t-1)(x_{t}+y_{t}+z_{t})+i} P_{c}^{\sum_{t=1}^{\ell+1} y_{t}} P_{d}^{\sum_{t=1}^{\ell+1} z_{t}}.
\end{split}
\end{equation*}
\noindent
Summing over $i=0,1,\ldots,\ell$ the result follows.
\end{proof}

%\begin{equation*}\label{eq: 1.1}
%\begin{split}
%P&(L_{n}=n)\\
%=&\sum_{i=0}^{l}\sum_{x_{1},\cdots,x_{l+1},y_{1},\cdots,y_{l+1},z_{1},\cdots,z_{l+1}}\Big[{\sum_{t=1}^{l+1} (x_{t}+y_{t}+z_{t})   \choose   x_{1},\cdots,x_{l+1},y_{1},\cdots,y_{l+1},z_{1},\cdots,z_{l+1}}\\
%&\times P_{a}^{\sum_{t=1}^{l+1} x_{t}} P_{b}^{\sum_{t=1}^{l+1} (t-1)(x_{t}+y_{t}+z_{t})+i} P_{c}^{\sum_{t=1}^{l+1} y_{t}} P_{d}^{\sum_{t=1}^{l+1} z_{t}}\Big] \\
%=&\sum_{i=0}^{l}\sum_{x_{1},\cdots,x_{l+1},y_{1},\cdots,y_{l+1},z_{1},\cdots,z_{l+1}}\Big[{\sum_{t=1}^{l+1} (x_{t}+y_{t}+z_{t})   \choose   x_{1},\cdots,x_{l+1},y_{1},\cdots,y_{l+1},z_{1},\cdots,z_{l+1}}\\
%&\times\Big(\frac{P_{a}}{P_{b}}\Big)^{\sum_{t=1}^{l+1} x_{t}} \Big(\frac{P_{c}}{P_{b}}\Big)^{\sum_{t=1}^{l+1} y_{t}} \Big(\frac{P_{d}}{P_{b}}\Big)^{\sum_{t=1}^{l+1} z_{t}} P_{b} ^{ \sum_{t=1}^{l+1} t(x_{t}+y_{t}+z_{t})+i}\Big]\\
%=&P_{b} ^{n}\sum_{i=0}^{l}\sum_{x_{1},\cdots,x_{l+1},y_{1},\cdots,y_{l+1},z_{1},\cdots,z_{l+1}}\Big[{\sum_{t=1}^{l+1} (x_{t}+y_{t}+z_{t})  \choose   x_{1},\cdots,x_{l+1},y_{1},\cdots,y_{l+1},z_{1},\cdots,z_{l+1}}\\
%&\times\Big(\frac{P_{a}}{P_{b}}\Big)^{\sum_{t=1}^{l+1} x_{t}} \Big(\frac{P_{c}}{P_{b}}\Big)^{\sum_{t=1}^{k} y_{l+1}} \Big(\frac{P_{d}}{P_{b}}\Big)^{\sum_{t=1}^{l+1} z_{t}}\Big]\\
%\end{split}
%\end{equation*}

\subsection{Distributions of maximum length and number of times it appears}
Let $L_{n}$ denote the maximum length of runs of B's and $NL_{n}$ the number of times a run of length $L_{n}$ appears in a sequence of size $n$. We establish the joint PMF of the random variables $L_{n}$ and $NL_{n}$ using combinatorial analysis.

\begin{theorem}
\label{thm 6}
The joint PMF of $L_{n}$ and $NL_{n}$ is given by
\begin{equation*}\label{eq: 1.1}
\begin{split}
P(L_{n}=\ell \wedge NL_{n}=x)=P_{b}^{n}\sum_{i=0}^{\ell}\sum_{\star}
\Bigg[&{x_{1}+\cdots +x_{\ell+1}+y_{1}+\cdots+y_{\ell+1}+z_{1}+\cdots+z_{\ell+1} \choose   x_{1},\ldots,x_{\ell+1},y_{1},\ldots,y_{\ell+1},z_{1},\ldots,z_{\ell+1}}\\
&\times
\left(\frac{P_{a}}{P_{b}}\right)^{x_{1}+\cdots +x_{\ell+1}} \left(\frac{P_{c}}{P_{b}}\right)^{y_{1}+\cdots+y_{\ell+1}} \left(\frac{P_{d}}{P_{b}}\right)^{z_{1}+\cdots+z_{\ell+1}}\Bigg],
\end{split}
\end{equation*}
\noindent
where the inner summation $\sum_{\star}$ is over all nonnegative integers $ x_{1}, \ldots, x_{\ell+1}, y_{1}, \ldots ,y_{\ell+1}, z_{1}, \ldots, z_{\ell+1}$ for which $\sum_{t=1}^{\ell+1}t(x_{t}+y_{t}+z_{t})=n-i,$ and satisfying at least a conditions
\noindent
\begin{equation*}\label{eq: 1.1}
\begin{split}
\left\{
  \begin{array}{ll}
    x_{\ell+1}\geq 1 \ \text{or} \\
    y_{\ell+1}\geq 1 \ \text{or} \\
    z_{\ell+1}\geq 1 \ \text{or} \\
    i=\ell\ \text{and}\ 0\leq i\leq \min(\ell, n-x(\ell+1)),\\
  \end{array}
\right.
\end{split}
\end{equation*}
\noindent
subject to
\noindent
\begin{equation*}\label{eq: 1.1}
\begin{split}
x_{\ell+1}+y_{\ell+1}+z_{\ell+1}=
\left\{
  \begin{array}{ll}
    x & \text{if $i\neq \ell,$} \\
    x-1 & \text{if $i=\ell$.}\\
  \end{array}
\right.
\end{split}
\end{equation*}

\end{theorem}

\begin{proof}
Let  ${\underbrace{B\cdots B}_{t-1}}A=O_{t}$,\quad ${\underbrace{B\cdots B}_{t-1}}C=J_{t}$,\quad ${\underbrace{B\cdots B}_{t-1}}D=T_{t}$,\ where $1\leq t \leq \ell+1$. A typical element of the event $\left\{L_{n}=\ell \wedge  NL_{n}=x\right\}$ is a sequence
\begin{equation*}\label{eq: 1.1}
\begin{split}
\cdots {\boxed{ R_{t}}} \cdots \overset{1}{\boxed{O_{\ell+1}\ or\ J_{\ell+1}\ or\ T_{\ell+1}}} \cdots  {\boxed{ R_{t}}}  \cdots\overset{x}{\boxed{O_{\ell+1}\ or\ J_{\ell+1}\ or\ T_{\ell+1}}} \cdots {\boxed{ R_{t}}} \cdots\underbrace{B\cdots B}_{i}
\end{split},
\end{equation*}
\noindent
where $0\leq i \leq \min(\ell,n-x(\ell+1))$, and ${\boxed{ R_{t}}}$ represents any string $O_{t}$, $J_{t}$, and $T_{t}$ appearing $x_{t}$, $y_{t}$, and $z_{t}$ times, respectively, in the sequence, satisfying
\noindent
 \begin{equation}\label{eq: 1.1}
\sum_{t=1}^{\ell+1}t(x_{t}+y_{t}+z_{t})=n-i,
\end{equation}
\noindent
subject to the condition $x_{\ell+1}+y_{\ell+1}+z_{\ell+1}=
\left\{
  \begin{array}{ll}
    x & \text{if $i\neq \ell,$} \\
    x-1 & \text{if $i=\ell$.}\\
  \end{array}
\right.$
\noindent
The number of different ways of arranging the sequence equals
\noindent
\begin{equation*}\label{eq:bn}
\begin{split}
{x_{1}+\cdots +x_{\ell+1}+y_{1}+\cdots+y_{\ell+1}+z_{1}+\cdots+z_{\ell+1} \choose   x_{1},\ldots,x_{\ell+1},y_{1},\ldots,y_{\ell+1},z_{1},\ldots,z_{\ell+1}}.
\end{split}
\end{equation*}
\noindent
Because of the independence of the trials, the probability of the above sequence is
\noindent
\begin{equation}\label{eq:bn}
\begin{split}
P_{a}^{x_{1}+\cdots +x_{\ell+1}} P_{c}^{y_{1}+\cdots +y_{\ell+1}}P_{d}^{z_{1}+\cdots +z_{\ell+1}}P_{b}^{(x_{2}+y_{2}+z_{2})+2(x_{3}+y_{3}+z_{3})+\cdots +l(x_{\ell+1}+y_{\ell+1}+z_{\ell+1})}.
\end{split}
\end{equation}
\noindent
The probability of the run of B's of length $i$ for $ 0\leq i\leq n-x(\ell+1))$ at the end of each of these sequences is $P_{b}^{i}$, which leads to the overall probability.
\noindent
\begin{equation*}\label{eq: 1.1}
\begin{split}
P_{a}^{x_{1}+\cdots +x_{\ell+1}} P_{c}^{y_{1}+\cdots +y_{\ell+1}}P_{d}^{z_{1}+\cdots +z_{\ell+1}}P_{b}^{(x_{2}+y_{2}+z_{2})+2(x_{3}+y_{3}+z_{3})+\cdots +\ell(x_{\ell+1}+y_{\ell+1}+z_{\ell+1})+i}.
\end{split}
\end{equation*}
Summing over $i=0,1,\ldots,n-x(\ell+1)$ the results follows.
\end{proof}

\subsection{Distribution of the smallest run length}
Let $M_{n}$ denote the minimum length of a run of B's in $n$ i.i.d.\ four state trials. We establish the PMF of the random variable $M_{n}$ using combinatorial analysis.

%$P(A)=p_{a}\quad P(B)=p_{b} \quad P(C)=p_{c} \quad P(D)=p_{d}, \quad p_{a}+p_{b}+p_{c}+p_{d}=1 $\\
%
%$x_{1},x_{s+1},\cdots,x_{n},y_{1},y_{s+1},\cdots,y_{n},z_{1},z_{s+1},\cdots,z_{n}$ is nonnegative integers  s.t. $(x_{1}+y_{1}+z_{1})+ (s+1)(x_{s+1}+y_{s+1}+z_{s+1})+\cdots+n(x_{n}+y_{n}+z_{n}) =n-i$, $i=0,s,s+1,\cdots,n$ satisfying at least one of the conditions $x_{s+1}\geq 1$ or $y_{s+1}\geq 1$ or $z_{s+1}\geq 1$ or $i=s$.
%
%Let  ${\underbrace{B\cdots B}_{t-1}}A=O_{t},\quad {\underbrace{B\cdots B}_{t-1}}C=J_{t}, {\underbrace{B\cdots B}_{t-1}}D=T_{t},  t=1,s+1,\cdots,n$
%
%$$\cdots \oplus \cdots \cdots \oplus \cdots\cdots \oplus \cdots\underbrace{B\ldots B}_{i}$$\\
%
%\begin{equation*}\label{eq: 1.1}
%\begin{split}
%P&(M_{n}=n)\\
%=&\sum_{i=0}^{n}\underset{(x_{1}+y_{1}+z_{1})+ (s+1)(x_{s+1}+y_{s+1}+z_{s+1})+\cdots+n(x_{n}+y_{n}+z_{n}) =n-i}{\mathop{\sum}_{x_{1},x_{s+1},\cdots,x_{n},y_{1},y_{s+1},\cdots,y_{n},z_{1},z_{s+1},\cdots,z_{n}}}\\
%&\Big[{(x_{1}+y_{1}+z_{1})+(x_{s+1}+y_{s+1}+z_{s+1})+\cdots+(x_{n}+y_{n}+z_{n})   \choose   x_{1},x_{s+1},\cdots,x_{n},y_{1},y_{s+1},\cdots,y_{n},z_{1},z_{s+1},\cdots,z_{n}}\\
% &\times P_{a}^{\sum_{t=1}^{l+1} x_{t}} P_{b}^{\sum_{t=1}^{l+1} (t-1)(x_{t}+y_{t}+z_{t})+i} P_{c}^{\sum_{t=1}^{l+1} y_{t}} P_{d}^{\sum_{t=1}^{l+1} z_{t}}\Big] \\
%\end{split}
%\end{equation*}
%

%%???????? ??y???

\begin{theorem}
\label{thm 7}
The PMF of $M_{n}$, for $0 \leq s\leq n$, is given by
\begin{equation*}\label{eq: 1.1}
\begin{split}
P(M_{n}=s)=\sum_{i}\sum_{\star}&\Bigg[{x_{1}+x_{s+1}+\cdots+x_{n}+y_{1}+y_{s+1}+\cdots+y_{n}+z_{1}+z_{s+1}+\cdots+z_{n}\choose x_{1},x_{s+1},\ldots,x_{n},y_{1},y_{s+1},\ldots,y_{n},z_{1},z_{s+1},\ldots,z_{n}}\\
&\times\left(\frac{P_{a}}{P_{b}}\right)^{x_{1}+x_{s+1}+\cdots+x_{n-s}} \left(\frac{P_{c}}{P_{b}}\right)^{y_{1}+y_{s+1}+\cdots+y_{n-s}} \left(\frac{P_{d}}{P_{b}}\right)^{z_{1}+z_{s+1}+\cdots+z_{n-s}}\Bigg],
\end{split}
\end{equation*}
where the summation $\sum_{\star}$ is over all nonnegative integers $x_{1}, x_{s+1}, \ldots, x_{n}, y_{1}, y_{s+1}, \ldots, y_{n} , z_{1}, z_{s+1}, \ldots, z_{n}$ such that $(x_{1}+y_{1}+z_{1})+ (s+1)(x_{s+1}+y_{s+1}+z_{s+1})+\cdots+n(x_{n}+y_{n}+z_{n}) =n-i$, where $i\in \{0, s, s+1, \ldots, n\},$ satisfying at least one of the conditions: $x_{s+1}\geq 1$, $y_{s+1}\geq 1$, $z_{s+1}\geq 1$, and $i=s$.
\end{theorem}
\begin{proof}
Let  ${\underbrace{B\cdots B}_{t-1}}A=O_{t}$,\quad ${\underbrace{B\cdots B}_{t-1}}C=J_{t}$,\quad ${\underbrace{B\cdots B}_{t-1}}D=T_{t}$,\ where $t=1,s+1,\ldots,n$. A typical element of the event $\left\{M_{n}=s\right\}$ is a sequence
\noindent
\begin{equation*}\label{eq: 1.1}
\begin{split}
\cdots {\boxed{ R_{t}}} \cdots\underbrace{B\cdots B}_{i}
\end{split},
\end{equation*}
\noindent
where ${\boxed{ R_{t}}}$ represents any of the strings $O_{t}$, $J_{t}$, and $T_{t}$ appearing $x_{t}$, $y_{t}$, and $z_{t}$ times, respectively, in the sequence, satisfying $(x_{1}+y_{1}+z_{1})+ (s+1)(x_{s+1}+y_{s+1}+z_{s+1})+\cdots+n(x_{n}+y_{n}+z_{n}) =n-i$, for $i\in \{0,s,s+1,\ldots,n\},$ and satisfying at least one of the conditions: $x_{s+1}\geq 1$, $y_{s+1}\geq 1$, $z_{s+1}\geq 1$, and $i=s$.
\noindent
The number of different ways of arranging the sequence equals
\noindent
\begin{equation}\label{eq:bn}
\begin{split}
{x_{1}+x_{s+1}+\cdots+x_{n}+y_{1}+y_{s+1}+\cdots+y_{n}+z_{1}+z_{s+1}+\cdots+z_{n}\choose x_{1},x_{s+1},\ldots,x_{n},y_{1},y_{s+1},\ldots,y_{n},z_{1},z_{s+1},\ldots,z_{n}}.
\end{split}
\end{equation}
\noindent
By the independence of trials, the probability of the above sequence is given by
\noindent
\begin{equation}\label{eq:bn}
\begin{split}
&P_{a}^{x_{1}+x_{s+1}+\cdots+x_{n-s}} P_{b}^{s(x_{s+1}+y_{s+1}+z_{s+1})+(s+1)(x_{s+2}+y_{s+2}+z_{s+2})+\cdots+(n-1)(x_{n}+y_{n}+z_{n})}\\ &\times P_{c}^{y_{1}+y_{s+1}+\cdots+y_{n-s}} P_{d}^{z_{1}+z_{s+1}+\cdots+z_{n-s}}.
\end{split}
\end{equation}
\noindent
The probability of a run of B's of length $i\in \{0,s,s+1,\ldots,n\}$ at the end of each of these sequences is $P_{b}^{i}$, which leads to the overall probability
\noindent
\begin{equation*}\label{eq: 1.1}
\begin{split}
&P_{a}^{x_{1}+x_{s+1}+\cdots+x_{n-s}} P_{b}^{s(x_{s+1}+y_{s+1}+z_{s+1})+(s+1)(x_{s+2}+y_{s+2}+z_{s+2})+\cdots+(n-1)(x_{n}+y_{n}+z_{n})+i}\\ &\times P_{c}^{y_{1}+y_{s+1}+\cdots+y_{n-s}} P_{d}^{z_{1}+z_{s+1}+\cdots+z_{n-s}}.
\end{split}
\end{equation*}
\noindent
Summing over $i=0,s,s+1,\ldots,n$ the result follows.
\noindent
\end{proof}

\subsection{Distributions of minimum length and number of times it appears}

Let $M_{n}$ denote the minimum length of a run of B's and $NM_{n}$ be the number of times a run of length $M_{n}$ appears in a sequence of size $n$. We establish the joint PMF of the random variables $M_n$ and $NM_{n}$ using combinatorial analysis.

\begin{theorem}
\label{thm 8}
The joint PMF of $M_n$ and $NM_{n}$, for $0\leq s\leq n$, $0\leq x\leq \left[\frac{n+1}{s+1}\right]$, is given by
\begin{equation*}\label{eq: 1.1}
\begin{split}
P(M_{n}=s \wedge  NM_{n}=x)=&P_{b}^{n}\sum_{i=0}^{l}\sum_{\star}
{x_{1}+x_{s+1}+\cdots +x_{n}+y_{1}+y_{s+1}+\cdots+y_{n}+z_{1}+z_{s+1}+\cdots+z_{n}   \choose   x_{1},x_{s+1},\ldots,x_{n},y_{1},y_{s+1},\ldots,y_{n},z_{1},z_{s+1},\ldots,z_{n}}     \\
&\times
\Big(\frac{P_{a}}{P_{b}}\Big)^{x_{1}+x_{s+1}+\cdots +x_{n}} \Big(\frac{P_{c}}{P_{b}}\Big)^{y_{1}+y_{s+1}+\cdots+y_{n}} \Big(\frac{P_{b}}{P_{d}}\Big)^{z_{1}+z_{s+1}+\cdots+z_{n}},
\end{split}
\end{equation*}
\noindent
where the inner summation $\sum_{\star}$ is over all nonnegative integers $x_{1}, x_{s+1}, \ldots, x_{n}, y_{1}, y_{s+1}, \ldots, y_{n}, z_{1}, z_{s+1}, \ldots, z_{n}$ for which $(x_{1}+y_{1}+z_{1})+(s+1)(x_{s+1}+y_{s+1}+z_{s+1})+\cdots +n(x_{n}+y_{n}+z_{n})=n-i$, for $i\in \{0, s, s+1, \ldots, n-x(s+1)\},$ while satisfying at least one of the conditions: $x_{s+1}\geq 1$, $y_{s+1}\geq 1$, $z_{s+1}\geq 1$, and $i=s,$ and subject to $x_{s+1}+y_{s+1}+z_{s+1}=
\left\{
  \begin{array}{ll}
    x & \text{if $i\neq s,$} \\
    x-1 & \text{if $i=s$.}\\
  \end{array}
\right.$
\end{theorem}

\begin{proof}
Let  ${\underbrace{B\cdots B}_{t-1}}A=O_{t}$,\quad ${\underbrace{B\cdots B}_{t-1}}C=J_{t}$,\quad ${\underbrace{B\cdots B}_{t-1}}D=T_{t}$,\ where $s+1\leq t\leq n$. A typical element of the event $\left\{M_{n}=s\wedge NM_{n}=x\right\}$ is a sequence
\begin{equation*}\label{eq: 1.1}
\begin{split}
\cdots {\boxed{ R_{t}}} \cdots \overset{1}{\boxed{O_{s+1}\ or\ J_{s+1}\ or\ T_{s+1}}} \cdots  {\boxed{ R_{t}}}  \cdots\overset{x}{\boxed{O_{s+1}\ or\ J_{s+1}\ or\ T_{s+1}}} \cdots {\boxed{ R_{t}}} \cdots\underbrace{B\cdots B}_{i}
\end{split},
\end{equation*}
\noindent
where $i\in \{0, s, s+1, \ldots, n-x(s+1)\}$, and ${\boxed{ R_{t}}}$ represents any string $O_{t}$, $J_{t}$, and $T_{t},$ appearing $x_{t}$, $y_{t}$ and $z_{t}$ times, respectively, in the sequence, satisfying
\noindent
 \begin{equation}\label{eq: 1.1}
(x_{1}+y_{1}+z_{1})+(s+1)(x_{s+1}+y_{s+1}+z_{s+1})+\cdots +n(x_{n}+y_{n}+z_{n})=n-i,
\end{equation}
\noindent
subject to the condition $x_{s+1}+y_{s+1}+z_{s+1}=
\left\{
  \begin{array}{ll}
    x & \text{if $i\neq s,$} \\
    x-1 & \text{if $i=s$.}\\
  \end{array}
\right.$
\noindent
The number of different ways of arranging the sequence equals
\noindent
\begin{equation*}\label{eq:bn}
\begin{split}
{x_{1}+x_{s+1}+\cdots +x_{n}+y_{1}+y_{s+1}+\cdots+y_{n}+z_{1}+z_{s+1}+\cdots+z_{n}   \choose   x_{1},x_{s+1},\ldots,x_{n},y_{1},y_{s+1},\ldots,y_{n},z_{1},z_{s+1},\ldots,z_{n}}.
\end{split}
\end{equation*}
\noindent
By the independence of trials, the probability of the above sequence is given by
\noindent
\begin{equation}\label{eq:bn}
\begin{split}
P_{a}^{x_{1}+x_{s+1}+\cdots +x_{n}} P_{c}^{y_{1}+y_{s+1}+\cdots+y_{n}}P_{d}^{z_{1}+z_{s+1}+\cdots+z_{n}}P_{b}^{s(x_{s+1}+y_{s+1}+z_{s+1})+\cdots +(n-1)(x_{n}+y_{n}+z_{n})}.
\end{split}
\end{equation}
\noindent
The probability of the run of B's of length $i\in \{0,s,s+1,\ldots,n\}$ at the end of each of these sequences is $P_{b}^{i}$, which leads to the overall probability
\noindent
\begin{equation*}\label{eq: 1.1}
\begin{split}
P_{a}^{x_{1}+x_{s+1}+\cdots +x_{n}} P_{c}^{y_{1}+y_{s+1}+\cdots+y_{n}}P_{d}^{z_{1}+z_{s+1}+\cdots+z_{n}}P_{b}^{s(x_{s+1}+y_{s+1}+z_{s+1})+\cdots +(n-1)(x_{n}+y_{n}+z_{n})+i}.
\end{split}
\end{equation*}
\noindent
Summing over $i=0,s,s+1,\ldots,n$ the result follows.
\noindent
\end{proof}

%
%{\appendix
%\section*{Appendix Title}

\section*{Acknowledgement}

The author would like to thank Dr. Tommy Rene Jensen whose comments led to significant improvements in this manuscript.


\begin{thebibliography}{999}
% Reference 1
%\bibitem[Author1(year)]{ref-journal}
%Author1, T. The title of the cited article. {\em Journal Abbreviation} {\bf 2008}, {\em 10}, 142--149.

%1
\bibitem[Aki, S. (1992)]{aki92}
Aki, S. (1992). Waiting time problems for a sequence of discrete random variables.
{\em Annals of the Institute of Statistical Mathematics}, {\bf 44(2)}, 363--378.


%2
\bibitem[Aki, S. (2019)]{aki19}
Aki, S. (2019). Waiting time for consecutive repetitions of a pattern and related distributions.
{\em Annals of the Institute of Statistical Mathematics}, {\bf 71(2)}, 307--325.

%3

\bibitem[Antzoulakos, D. L. (1999)]{antzoulakos99}
Antzoulakos, D. L. (1999). On waiting time problems associated with runs in Markov dependent trials.
{\em Annals of the Institute of Statistical Mathematics}, {\bf 51(2)}, 323--330.




%4
\bibitem[Antzoulakos, D. L. and Philippou,~A. N. (1997)]{antzoulakos97}
Antzoulakos, D. L., and Philippou, A. N. (1997). Probability distribution functions of succession quotas in the case of Markov dependent trials.
{\em Annals of the Institute of Statistical Mathematics}, {\bf 49(3)}, 531--539.


%5

\bibitem[Arapis, A. N., Makri, F. S., and Psillakis,~Z.~M. (2018)]{arapis18}
Arapis, A. N., Makri, F. S. and Psillakis, Z. M. (2018). Distributions of statistics describing concentration of runs in non homogeneous Markov-dependent trials.
{\em Communications in Statistics-Theory and Methods}, {\bf 47(9)}, 2238--2250.


%6
\bibitem[Balakrishnan, N. and Koutras, M.~V. (2003)]{balakoutras03}
N. Balakrishnan, M.V. Koutras. (2003). Runs and scans with applications,
        John Wiley \& Sons, New York.



%7



\bibitem[Balasubramanian. K. (1993)]{balasubramanian93}
K. Balasubramanian, R. Viveros, N. Balakrishnan (1993) Sooner and later waiting time problems for Markovian Bernoulli trials.
{\em Statistics \& Probability Letters}, {\bf 18(2)}, 153--161.



%8

\bibitem[Chang, GJLR and In-Hang, Chung and Cui, Lirong and Hwang, FK. (2000)]{chang2000}
Chang, GJLR and In-Hang, Chung and Cui, Lirong and Hwang, FK. (2000). Reliabilities of consecutive-k systems,
        Springer Science \& Business Media.


\bibitem[Chao, MT and Fu, JC and Koutras, MV. (1995)]{chao1995}
Chao, MT and Fu, JC and Koutras, MV. (1995). Survey of reliability studies of consecutive-k-out-of-n: F and related systems. {\em IEEE Transactions on reliability}, {\bf 44(1)}, 120--127.


%9
\bibitem[Ebneshahrashoob, M. and Sobel, M. (1990)]{ebneshahrashoob90}
Ebneshahrashoob, M. and Sobel, M. (1990). Sooner and later waiting time problems for Bernoulli trials: frequency and run quotas.
{\em Statistics \& Probability Letters}, {\bf 9(1)}, 5--11.


%10
\bibitem[Eryilmaz, S. (2010)]{ser10}
Eryilmaz, S. (2010). Review of recent advances in reliability of consecutive k-out-of-n and related systems. {\em Proceedings of the Institution of Mechanical Engineers, Part O: Journal of Risk and Reliability}, {\bf 224(3)}, 225--237.

%11
\bibitem[Eryilmaz, S. (2018)]{serkan18}
Eryilmaz, S. (2018). On success runs in a sequence of dependent trials with a change point.
{\em Statistics \& Probability Letters}, {\bf 132}, 91--98.



%12
\bibitem[Feller, W. (1968)]{feller68}
Feller, W. (1968). An introduction to probability theory and its applications. {\em Vol.\rm{1}},
        John Wiley \& Sons, New York.


%13
\bibitem[Fu, J. C. and  Lou, W. W. (2003)]{fuwendy03}
Fu, J. C. and  Lou, W. W. (2003). Distribution theory of runs and patterns and its applications: a finite Markov chain imbedding approach,
        World Scientific.



%14
\bibitem[Inoue, K. and Aki, S. (2014)]{inoue14}
Inoue, K. and Aki, S. (2014). On sooner and later waiting time distributions associated with simple patterns in a sequence of bivariate trials.
{\em Metrika}, {\bf 77(7)}, 895--920.


%15
\bibitem[Kim, S., Park, C. and Oh, J. (2013)]{kim13}
Kim, S., Park, C. and Oh, J. (2013). On waiting time distribution of runs of ones or zeros in a Bernoulli sequence.
{\em Statistics \& Probability Letters}, {\bf 83(1)}, 339--344.


%16
\bibitem[Kong, Y. (2019)]{kong19}
Kong, Y. (2019). Joint distribution of rises, falls, and number of runs in random sequences. {\em Communications in Statistics-Theory and Methods}, {\bf 48(3)}, 493--499.

%17
\bibitem[Koutras, M. V. (1996)]{koutras96}
Koutras, M. V. (1996). On a waiting time distribution in a sequence of Bernoulli trials.
{\em Annals of the Institute of Statistical Mathematics}, {\bf 48(4)}, 789--806.

%18
\bibitem[Kuo, W and Zuo, M. (2003)]{kuo2003}
Kuo, Way and Zuo, Ming J. (2003). Optimal reliability modeling: principles and applications,
        John Wiley \& Sons, New York.


%19
\bibitem[Makri, F. S., Psillakis, Z. M. and Arapis, A. N. (2019)]{makri19}
Makri, F. S., Psillakis, Z. M. and Arapis, A. N. (2019). On the concentration of runs of ones of length exceeding a threshold in a Markov chain.
{\em Journal of Applied Statistics}, {\bf 46(1)}, 85--100.



%20
\bibitem[Mood, A. M. (1940)]{mood40}
Mood, A. M. (1940). The distribution theory of runs. {\em The Annals of Mathematical Statistics}, {\bf 11(4)}, 367--392.

%21

\bibitem[Philippou, A. N. (1984)]{philippou84}
Philippou, A. N. (1984). The negative binomial distribution of order k and some of its properties.
{\em Biometrical Journal}, {\bf 26(7)}, 789-794.



%22
\bibitem[Philippou, A. N. and Muwafi, A. A. (1982)]{philippou82}
Philippou, A. N. and Muwafi, A. A. (1982). Waiting for the kth consecutive success and the fibonacci sequence of order k.
{\em The Fibonacci Quartly}, {\bf 20(1)}, 28--32.



%23

\bibitem[Philippou, A. N., Georghiou, C., and Philippou, G. N. (1983)]{philippou83}
Philippou, A. N., Georghiou, C. and Philippou, G. N. (1983). A generalized geometric distribution and some of its properties.
{\em Statistics \& Probability Letters}, {\bf 1(4)}, 171--175.



%24
\bibitem[Philippou, A. N. and Makri,~F.~S. (1986)]{philippou86}
Philippou, A. N. and Makri, F. S.(1986). Successes, runs and longest runs.
{\em Statistics \& Probability Letters}, {\bf 4(2)}, 101--105.




%25
\bibitem[Sen, K., Agarwal, M. L. and Chakraborty, S. (2002)]{sen02}
Sen, K., Agarwal, M. L. and Chakraborty, S. (2002). Lengths of runs and waiting time distributions by using Polya-Eggenberger sampling scheme. {\em Studia Scientiarum Mathematicarum Hungarica}, {\bf 39}, 309--332.



%26
\bibitem[Triantafyllou, Ioannis S. (2015)]{triantafyllou2015}
Triantafyllou, Ioannis S. (2015). Consecutive-type reliability systems: an overview and some applications. {\em Journal of Quality and Reliability Engineering}, {\bf 2015}.

\end{thebibliography}
\end{document}